\documentclass[12pt]{article}
\usepackage{amsmath,amsfonts, amssymb, graphicx}
\title{ Relatively Hyperbolic Groups have Semistabile Fundamental Group at Infinity} 
\author{M. Mihalik and E. Swenson\footnote{This research was supported in part by Simons Foundation Grant 209403}}
\newtheorem{theorem}{Theorem}[section]
\newtheorem{proposition}[theorem]{Proposition}
\newtheorem{lemma}[theorem]{Lemma}

\newcounter{remarknum}
\newenvironment{remark}{\addvspace{12pt}\refstepcounter{remarknum}
\noindent{\bf Remark \arabic{remarknum}.}}{\par\addvspace{12pt}}
\newenvironment{proof}{\addvspace{12pt}\noindent{\bf Proof:}}{
$\Box$\par\addvspace{12pt}}
\newcounter{examplenum}
\newenvironment{example}{\addvspace{12pt}\refstepcounter{examplenum}
\noindent{\bf Example \arabic{examplenum}.}}{\par\addvspace{12pt}}

\newcounter{definitionnum}
\newenvironment{definition}{\addvspace{12pt}\refstepcounter{definitionnum}
\noindent{\bf Definition \arabic{definitionnum}.}}{\par\addvspace{12pt}}

\date{\today}
\begin{document}
\maketitle

\begin{abstract} 
Suppose $G$ is a 1-ended finitely generated group that is hyperbolic relative to {\bf P} a finite collection of 1-ended finitely generated proper subgroups of $G$. Our main theorem states that if the boundary $\partial (G,{\bf P})$ has no cut point, then $G$ has  semistable fundamental group at $\infty$. Under mild conditions on $G$ and the members of {\bf P} the 1-ended hypotheses and the no cut point condition can be eliminated to obtain the same semistability conclusion. We give an example that shows our main result is somewhat optimal. Finally, we improve a ``double dagger" result of F. Dahmani and D. Groves. 
\end{abstract}

\section{Introduction}\label{Intro}
In this paper, we are interested in the asymptotic behavior of relatively hyperbolic groups. We consider a property of finitely presented groups that has been well studied for over 40 years called semistable fundamental group at $\infty$. 
It is unknown at this time, whether or not all finitely presented groups have semistable fundamental group at $\infty$. The finitely presented group $G$ satisfies a weaker condition called semistable first homology at $\infty$ if and only if  $H^2(G:\mathbb ZG)$ is free abelian (see \cite{MR787167}) and it is also a long standing problem if this is always the case. Our main interest is in showing certain relatively hyperbilic groups have semistable fundamental group at $\infty$. The work of B. Bowditch \cite{Bow99B} and G. Swarup \cite{Swarup} show that if $G$ is a 1-ended word hyperbolic group then $\partial G$, the boundary of $G$, has no (global) cut point. By work of M. Bestvina and G. Mess \cite{BM91}, this is equivalent to $\partial G$ being locally connected. It was pointed out by R. Geoghegan that $G$ has semistable fundamental group at $\infty$ if and only if $\partial G$ has the shape of a locally connected continuum (see \cite{DydakS} for a proof of this fact). In particular, all 1-ended word hyperbolic groups have semistable fundamental group at $\infty$. Relatively hyperbolic groups are a much studied generalization of hyperbilic groups. In section \ref{SB}, we define what it means for a finitely generated 1-ended group to be hyperbolic relative to a finite collection of finitely generated subgroups.  Relatively hyperbolic groups have a well-defined boundary.  While 1-ended relatively hyperbolic groups nearly always have locally connected boundary, unlike the case for hyperbolic groups, there may also be cut points in this boundary. 

\begin{theorem} (\cite {Bow01}, Theorem 1.5) Suppose $(G, {\bf P})$ is relatively hyperbolic, $G$ is 1-ended and each $P\in {\bf P}$ is finitely presented, does not contain an infinite torsion group, and is either 1 or 2-ended, then $\partial (G,{\bf P})$ is locally connected.
 \end{theorem}
 
Our main result follows. Note that there is no semistability hypothesis on the peripheral subgroups $P_i$.
\begin{theorem} \label{Main}
Suppose $G$ is a 1-ended finitely generated group that is hyperbolic relative to a collection of 1-ended finitely generated proper subgroups ${\bf P}=\{P_1,\ldots, P_n\}$. If $\partial (G,{\bf P})$ has no cut point, then $G$ is semistable at $\infty$.
\end{theorem}

We will list several results in the literature which, in conjunction with our theorem,  imply that in many cases we can drop the 1-ended hypotheses and the no cut point hypothesis and still obtain the same semistability conclusion. 
If a group is $G$ is finitely presented and hyperbolic relative to a finite collection of finitely generated subgroups $P_i$, then the $P_i$ are finitely presented as well (see \cite {DG13} or for a more general result \cite{DGO17}, Theorem 2.11). 
The following result can be derived from work of C. Drutu and M. Sapir \cite{DS05} (Corollary 1.14) or D. Osin \cite{Osin06}. 

\begin{theorem} \label{split1} 
Suppose $G$ is finitely generated and hyperbolic relative to the finitely generated groups $P_1,\ldots, P_n$. If $\mathcal P_i$ is a finite graph of groups decomposition of $P_i$ such that each edge group of $\mathcal P_i$ is finite, then $G$ is also hyperbolic relative to $\{P_1,\ldots, P_{i-1}, P_i,\ldots, P_n\}\cup V(\mathcal P_i)$ where $V(\mathcal P_i)$ is the set of vertex groups of $\mathcal P_i$. 
\end{theorem}

If a peripheral subgroup $P_i$ is either finite or 2-ended, it may be removed from the collection of peripheral subgroups and $G$ remains hyperbolic relative to the remaining subgroups. Recall that a finitely generated group is {\it accessible} if it has a finite graph of groups decomposition with each edge group finite and each vertex group either finite or 1-ended. By M. Dunwoody's accessibility theorem \cite{Dun85}, all (almost) finitely presented groups are accessible. In particular, requiring the $P_i$ in our theorem to be 1-ended is only a mild restriction. 

The following results about splittings are all due to B. Bowditch.

\begin{proposition} \label {P3.1} (\cite{Bow12}, Propositions 10.1-3). 
Let $(G,{\bf P})$ be a relatively hyperbolic group. Its boundary $\partial (G,{\bf P})$ is disconnected if and only if $G$ splits nontrivially over a finite group relative to ${\bf P}$. In this case, every vertex group is hyperbolic relative to the parabolic subgroups it contains.
In particular, if the parabolic subgroups of $G$ are all 1-ended, then $G$ is 1-ended if and only if $\partial G$ is connected.
\end{proposition}

A splitting is {\it relative to {\bf P}} if each element of ${\bf P}$ is conjugate into a vertex group of the spllitting. 

\begin{definition} (\cite{Bow01}). Let $(G,{\bf P})$ be a relatively hyperbolic group. A {\it peripheral splitting} of $G$ is a representation of $G$ as a finite bipartite graph of groups where {\bf P} consists precisely of the (conjugacy classes of) vertex groups of one color. A peripheral splitting is a refinement of another if there is a color preserving folding of the first into the second.
\end{definition}

The following organization of results of Bowditch appears in \cite{DG08}.

\begin{theorem} \label{B3.3} (B. Bowditch) 
Let (G, {\bf P}) be a relatively hyperbolic group. Assume that $\partial G$ is connected.

(1) (\cite{Bow99B}, Theorem 0.2) If every maximal parabolic subgroup of $G$ is (1 or 2)-ended, finitely presented, and without infinite torsion subgroup, then, every global cut point of $\partial G$ is a parabolic fixed point.

(2) (\cite{Bow99A}, Theorem 1.2) If there is a global cut point that is a parabolic point, then $G$ admits a proper peripheral splitting.

(3) (\cite{Bow01}, Theorem 1.2) If $G$ admits a proper peripheral splitting, then $\partial G$ admits a global cut point.
\end{theorem}

It is established in  (\cite{Bow01}, Theorem 1.3) that if $\partial G$ is connected, then any non-peripheral vertex group of a peripheral splitting also has connected boundary and is hyperbolic relative to its adjacent edge groups.

Bowditch also proves an accessibility result.
\begin{theorem} ( \cite{Bow01}, Theorem 1.4) \label{Acc}
Suppose that $G$ is relatively hyperbolic with connected boundary. Then $G$ admits a (possibly trivial) peripheral splitting which is maximal in the sense that it is not a refinement of any other peripheral splitting.
 \end{theorem}
 
The following result of M. Mihalik and S. Tschantz meshes well with the accessibility result of Bowditch. Notice that there is no restriction on the number of ends of any of the groups $G$, the vertex groups or the edge groups in this result. 

\begin{theorem} \cite{MT1992} \label{split} 
Suppose $\mathcal G$ is a finite graph of groups decomposition of the finitely presented group $G$ where each edge group is finitely generated and each vertex group is finitely presented with semistable fundamental group at $\infty$. Then $G$ has semistable fundamental group at $\infty$. 
\end{theorem}

\begin{example} \label{counter}
There are obvious limitations to our results. If $G$ is a 1-ended finitely generated group that is hyperbolic relative to an inaccessible group $P$ then our results do not apply. 

Suppose $B$ is a hyperbolic surface group with infinite cyclic subgroup $C$ and $A$ is an arbitrary 1-ended group with infinite cyclic subgroup $C'$. The group $G=A\ast _{C=C'}B$ is a finitely presented 1-ended group. By (\cite{D03}, Theorem 0.1) $G$ is hyperbolic relative to $A$.  The boundary $\partial (G, A)$ has a parabolic cut point. It would be unreasonable to expect the theory of relatively hyperbolic groups to imply that $G$ is semistable at $\infty$ in this case. Such a result would imply that (the arbitrary 1-ended group) $A$ is ``nearly" semistable at $\infty$. In fact, if $A$ contained an infinite cyclic group whose two ends were ``properly homotopic in  $A$" (any copy of $\mathbb Z\times \mathbb Z$ in $A$ would suffice), then in fact the semistabliity of $G$ would imply the semistability of $A$.
\end{example}

The boundaries of the non-peripheral vertex groups of a maximal splitting guaranteed by Theorem \ref{Acc} do not contain parabolic cut points and it is conjectured that they contain no global cut points of any sort. If that is the case, and $G$ is a finitely presented group that is hyperbolic relative to a finite collection {\bf P} of finitely presented subgroups then the only thing that would prevent our result (in conjunction with those listed in this section) from implying that $G$ is semistable at $\infty$ is the possibility of running into an inaccessible peripheral subgroup or a non-semistable  peripheral vertex group of the maximal splitting. If $G$ and the members of {\bf P} are all accessible, then Theorems \ref{split1} and \ref{P3.1} reduce to the case where $G$ and the members of {\bf P} are 1-ended. If the edge groups of non-peripheral vertex groups of a maximal peripheral splitting of $G$ are finitely presented, 1-ended and without infinite torsion subgroups, then they have no global cut points and Theorem \ref{Main} implies each of these vertex groups is semistable at $\infty$. Then Theorem \ref{split} implies $G$ is semistable at $\infty$.


 
All of our work is done in Groves-Manning space, a Gromov hyperbolic, locally finite 2-complex on which our relatively hyperbolic group acts by isometries (see \S \ref{GM}). It follows from (\cite{Bow12}, \S 6 and \S 9) that the Bowditch boundary agrees with the Gromov boundary of Groves-Manning space.  One of the main tools of our proof is a ``double dagger" result that appears as Lemma 4.2 of \cite{DG08}. This result is analogous to the original double dagger result of \cite{BM91}. The ``replacement" path guaranteed by Lemma 4.2 of \cite{DG08} is only guaranteed to have image in the Groves-Manning space $X$. We are able to improve this result (see Theorem \ref{DD+}) to provide a replacement path in $X_m$, a fixed subspace of $X$. 
 
If $G$ is finitely generated and {\bf P} is a finite collection of subgroups of $G$ such that $G$ is hyperbolic relative to ${\bf P}$ and $\partial (G,{\bf P})$ is locally connected then (by R. Geoghegan's observation) the corresponding Groves-Manning space can be shown to have semistable fundamental group at $\infty$ (a fact that is not used in our work). The base space $Y$ in $X$ is a universal cover of a finite complex with fundamental group $G$ and our goal in this paper is to show that $Y$ (and hence $G$) has semistable fundamental group at $\infty$. It is straightforward to show that when $Y$ has semistable fundamental group at $\infty$, the neighborhoods $X_m$ of $Y$ in $X$ also have semistable fundamental group at $\infty$ (see Theorem \ref{XKSS}).

 The remainder of the paper is organized as follows: In \S \ref{CAT} we list a few related results about the asymptotic behavior  of CAT(0) groups. In \S \ref{SB} we give several equivalent definitions of semistability that are used in our proofs. In \S \ref{GM} we define the Groves-Manning space and list three results from \cite{GMa08}. This space is where all of our technical work takes place. The main lemmas \ref{CaseN1} and \ref{CaseN2} are proved in \S \ref{ML}. Also an elementary, but critical observation, Lemma \ref{NoLinDeadEnd}, about ``dead ends" in a general finitely generated group is proved. Section \ref{MT} contains the proof of the main theorem and a proof that the subspaces $X_m$ of the Groves-Manning space $X$ have semistable fundamental group at $\infty$ (under the hypotheses of the main theorem).
 
 Section \S \ref{SDD} contains an argument that uses the main theorem and improves the double dagger result of Dahmani-Groves. Instead of producing ``far out" paths in the Groves-Manning space $X$, we produce ``far out" paths in $X_m$. Finally in section \S \ref{FC} a result that seems far more general than one of our key lemmas (Lemma \ref{CaseN1}) is proved. 
 
 {\it Acknowledgements:} We wish to thank Daniel Groves, Matthew Haulmark, Chris Hruska, Jason Manning and Denis Osin for helpful conversations.
 
 \section{Complimentary CAT(0) Results} \label {CAT} 
 The CAT(0) groups with isolated flats, are intuitively the CAT(0) groups that are closest to hyperbolic groups without actually being hyperbolic. As an application to their main theorem C. Hruska and K. Ruane prove: 
 
 \begin{theorem} \label{HR} 
 (\cite{HR17}, Theorem 1.2) Any CAT(0) group with isolated flats has semistable fundamental group at $\infty$. 
 \end{theorem}
 CAT(0) groups with isolated flats are hyperbolic relative to their maximal abelian subgroups and this fact plays a critical role in their argument. 
 
 While it is known that a boundary of a 1-ended CAT(0) group cannot have a global cut point (P. Papasoglu and E. Swenson \cite{PS09} and E. Swenson \cite{Sw}) it is unknown if all CAT(0) groups are semistable at $\infty$ (see \cite{GS17}). Let $F_2$ be the free group of rank 2. The CAT(0) group $F_2\times F_2$ is not hyperbolic with respect to any collection of subgroups (see \cite{DS05}). 

\section{Semistability Background} \label{SB} 

The best reference for the notion of semistable fundamental group at $\infty$ is \cite{G} and we use this book as a general reference throughout this section. While semistability makes sense for multiple ended spaces, we are only interested in 1-ended spaces in this article. Suppose $K$ is a 
1-ended locally finite connected CW complex. A {\it ray} in $K$ is a continuous map $r:[0,\infty)\to K$. A continuous map $f:X\to Y$ is {\it proper} if for each compact set $C$ in $Y$, $f^{-1}(C)$ is compact in $X$. 
The space $K$ has {\it semistable fundamental group at $\infty$} if any two proper rays in $K$ 
are properly homotopic. 
Suppose  $C_0, C_1,\ldots $ is a collection of compact subsets of a locally finite complex $K$ such that $C_i$ is a subset of the interior of $C_{i+1}$ and $\cup_{i=0}^\infty C_i=K$, and $r:[0,\infty)\to K$ is proper, then $\pi_1^\infty (K,r)$ is the inverse limit of the inverse system of groups:
$$\pi_1(K-C_0,r)\leftarrow \pi_1(K-C_1,r)\leftarrow \cdots$$
This inverse system is pro-isomorphic to an inverse system of groups with epimorphic bonding maps if and only if $K$ has semistable fundamental group at $\infty$.  Semistable fundamental group at $\infty$ is an invariant of proper homotopy type and quasi-isometry type. When $K$ is 1-ended with semistable fundamental group at $\infty$, $\pi_1^\infty (K,r)$ is independent of proper base ray $r$ (in direct analogy with the fundamental group of a path connected space being independent of base point). 
There are a number of equivalent forms of semistability. The equivalence of the conditions in the next theorem is discussed \cite{CM2}. 
\begin{theorem}\label{ssequiv} (see Theorem 3.2, \cite{CM2}) \label{EquivSS} 
Suppose $K$ is a connected 1-ended  locally finite CW-complex. Then the following are equivalent:
\begin{enumerate}
\item $K$ has semistable fundamental group at $\infty$.
\item Suppose $r:[0,\infty )\to K$ is a proper base ray. Then for any compact set $C$, there is a compact set $D$ such that for any third compact set $E$ and loop $\alpha$ based on $r$ and with image in $K-D$, $\alpha$ is homotopic $rel\{r\}$ to a loop in $K-E$, by a homotopy with image in $K-C$. 
\item For any compact set $C$ there is a compact set $D$ such that if $r$ and $s$ are proper rays based at $v$ and with image in $K-D$, then $r$ and $s$ are properly homotopic $rel\{v\}$, by a proper homotopy in $K-C$. 
\end{enumerate}
\end{theorem}

If $G$ is a finitely presented group and $X$ is  a finite connected complex with $\pi_1(X)=G$ then $G$ is {\it semistable at} $\infty$ if the universal cover of $X$ has semistable fundamental group at $\infty$. This definition only depends on $G$ and it is unknown if all finitely presented groups are semistable at $\infty$.  

In order to prove certain finitely presented groups have semistable fundamental group at $\infty$, the notion of semistable fundamental group at $\infty$ for a finitely generated group was introduced in  \cite{M4}. Our main interest here is also finitely presented groups, but M. Haulmark and C. Hruska have elementary examples of finitely presented groups that decompose according to Bowditch's accessibility result (Theorem {\ref{Acc}) with  vertex groups that are finitely generated and not finitely presented. In that situation our main theorem comes into play.  Suppose $S$ is a finite generating set for a finitely generated group $G$. Let $\Gamma(G,S)$ be the Cayley graph of $(G,S)$. If there are a finite set of relations $\mathcal R$ of $G$ such that the space resulting from attaching 2-cells to $\Gamma$ (one at each vertex for each $R\in \mathcal R$) produces a space that has semistable fundamental group at $\infty$, then $G$ is said to have semistable fundamental group at $\infty$. We first prove our main theorem in the finitely presented case. An elementary adjustment provides the finitely generated version. 

\section{Groves-Manning Space} \label{GM}

D. Groves and J. Manning \cite{GMa08} construct a locally finite Gromov hyperbolic space $X$ from a finitely generated group $G$ and a collection $\mathcal P$ of finitely generated subgroups. The following definitions are directly from \cite{GMa08}

\begin{definition}  Let $\Gamma$ be any 1-complex. The combinatorial horoball based on $\Gamma$,
denoted $\mathcal H(\Gamma)$, is the 2-complex formed as follows:

{\bf A)} $\mathcal H^{(0)} =\Gamma (0) \times (\{0\}\cup \mathbb N)$

{\bf B)} $\mathcal H^{(1)}$ contains the following three types of edges. The first two types are
called horizontal, and the last type is called vertical.

(B1) If $e$ is an edge of $\Gamma$ joining $v$ to $w$ then there is a corresponding edge
$\bar e$ connecting $(v, 0)$ to $(w, 0)$.

(B2) If $k > 0$ and $0 < d_{\Gamma}(v,w) \leq 2^k$, then there is a single edge connecting
$(v, k)$ to $(w, k)$.

(B3) If $k\geq 0$ and $v\in \Gamma^{(0)}$, there is an edge joining $(v,k)$ to $(v,k+1)$.

{\bf C)} $\mathcal H^{(2)}$ contains three kinds of 2-cells:

(C1) If $\gamma \subset  \mathcal  H^{(1)}$ is a circuit composed of three horizontal edges, then there
is a 2-cell (a horizontal triangle) attached along $\gamma$.

(C2) If $\gamma \subset \mathcal H^{(1)}$ is a circuit composed of two horizontal edges and two
vertical edges, then there is a 2-cell (a vertical square) attached along $\gamma$. 

(C3) If $\gamma\subset  \mathcal H^{(1)}$ is a circuit composed of three horizontal edges and two
vertical ones, then there is a 2-cell (a vertical pentagon) attached along $\gamma$, unless $\gamma$ is the boundary of the union of a vertical square and a horizontal triangle.
\end{definition}

\begin{definition} Let $\Gamma$ be a graph and $\mathcal H(\Gamma)$ the associated combinatorial horoball. Define a depth function
$$\mathcal D : \mathcal H(\Gamma) \to  [0, \infty)$$
which satisfies:

(1) $\mathcal D(x)=0$ if $x\in \Gamma$,

(2) $\mathcal D(x)=k$ if $x$ is a vertex $(v,k)$, and

(3) $\mathcal D$ restricts to an affine function on each 1-cell and on each 2-cell.
\end{definition}

\begin{definition} Let $\Gamma$ be a graph and $\mathcal H = \mathcal H(\Gamma)$ the associated combinatorial horoball. For $n \geq 0$, let $\mathcal H_n \subset \mathcal H$ be the full sub-graph with vertex set $\Gamma ^{(0)} \times \{0,\ldots ,N\}$, so that $\mathcal H_n=\mathcal D^{-1}[0,n]$.  Let $\mathcal H^n=\mathcal D^{-1}[n,\infty)$ and $\mathcal H(n)=\mathcal D^{-1}(n)$.
\end{definition}

\begin{lemma} \label{GM3.10} 
(\cite{GMa08}, Lemma 3.10) Let $\mathcal H(\Gamma)$ be a combinatorial horoball. Suppose that $x, y \in \mathcal H(\Gamma)$ are distinct vertices. Then there is a geodesic $\gamma(x, y) = \gamma(y, x)$ between $x$ and $y$  which consists of at most two vertical segments and a single horizontal segment of length at most 3.

Moreover, any other geodesic between $x$ and $y$ is Hausdorff distance at most 4 from this geodesic.
\end{lemma}

\begin{definition} Let $G$ be a finitely generated group, let $\mathcal P = \{P_1, \ldots , P_n\}$ be a (finite) family of finitely generated subgroups of $G$, and let $S$ be a generating set for $G$ containing generators for each of the $P_i$.   For each $i\in \{1,\ldots ,n\}$, let $T_i$ be a left transversal for $P_i$ (i.e. a collection of representatives for left cosets of $P_i$ in $G$ which contains exactly one element of each left coset).

For each $i$, and each $t \in T_i$, let $\Gamma_{i,t}$  be the full subgraph of the Cayley graph $\Gamma (G,S)$ which contains $tP_i$. Each $\Gamma_{i,t}$ is isomorphic to the Cayley graph of $P_i$ with respect to the generators $P_i \cap  S$. Then define
$$X =\Gamma (G,S)\cup (\cup \{\mathcal H(\Gamma_{i,t})^{(1)} |1\leq i\leq n,t\in T_i\}),$$
where the graphs $\Gamma_{i,t} \subset  \Gamma(G,S)$ and $\Gamma_{i,t} \subset  \mathcal H(\Gamma_{i,t})$ are identified in the obvious way.
\end{definition}
We call the space $X$ a {\it Groves-Manning} space for  $G$, $\mathcal P$ and $S$.  
The next result shows the Groves-Manning spaces are fundamentally important spaces.  We prove our results in these spaces. 

\begin{theorem} \label{GM3.25} 
(\cite{GMa08}, Theorem 3.25)
Suppose that $G$ is a finitely generated group and $\mathcal P=\{P_1,\ldots, P_n\}$ is a finite collection of finitely generated subgroups of $G$. Let $S$ be a finite generating set for $G$ containing generating sets for the $P_i$.  A  Groves-Manning space $X$ for $G$, $\mathcal P$ and $S$ is hyperbolic if and only if  $G$ is hyperbolic with respect to $\mathcal P$.
\end{theorem}

We make some minor adjustments. Assume $G$ is finitely presented and hyperbolic with respect to the subgroups $\mathcal P=\{P_1,\ldots, P_n\}$ and $S$ is a finite generating set for $G$ containing generating sets for the $P_i$. For a finite presentation $\mathcal A$ of $G$ with respect to $S$,  let $Y$ be the Cayley 2-complex for $\mathcal A$. So $Y$ is simply connected with 1-skeleton $\Gamma(G,S)$, and the quotient space $G/Y$  has fundamental group $G$. Let $X$ be a Groves-Manning space for $G$, $\mathcal P$ and $S$. Replace the Cayley graph $\Gamma(G,S)$ in $X$ with $Y$ in the obvious way. For $g\in G$ and $i\in\{1,\ldots,n\}$ we call $gP_i$ a {\it peripheral coset}. The depth functions on the horoballs over the peripheral cosets extend to $X$. So that
$$ \mathcal D:X\to [0,\infty)$$ 
where $\mathcal D^{-1}(0)=Y$ and for each horoball $H$ (over a peripheral coset) we have $H\cap \mathcal D^{-1}(m)=H(m)$, $H\cap \mathcal D^{-1}[0,m]=H_m$ and $H\cap \mathcal D^{-1}[m,\infty)=H^m$. We call each $H^m$ an {\it $m$-horoball}.


\begin{lemma} \label{geo} (\cite{GMa08}, Lemma 3.26) 
If Groves-Manning space $X$ is $\delta$-hyperbolic, then the $m$-horoballs of $X$ are convex for all $m\geq \delta$. 
\end{lemma} 

Given two points $x$ and $y$ in a horoball $H$, there is a shortest path in $H$ from $x$ to $y$ of the from $(\alpha, \tau,\beta)$ where $\alpha$ and $\beta$ are vertical and $\tau$ is horizontal of length $\leq 3$. Note that if $\alpha$ is non-trivial and ascending and $\beta$ is non-trivial and decending, then $\tau$ has length either 2 or 3. 

Let $\ast$ be the identity vertex of $Y\subset X$ and $d$ the edge path distance in $X$. 
If $H$ is a horoball, and $x,y\in H(m)$ then let $d^{m}(x,y)$ be the length of the shortest edge path from $x$ to $y$ in $H(m)$.  For $z\in H(m)$, let $\hat B_n(z)$ be the set of all vertices $v$ of $H(m)$ such that $d^{m}(v,z)\leq n$.

\section{The Main Lemmas} \label{ML} 
In this section we prove the two lemmas (\ref{CaseN1} and \ref{CaseN2}) that provide the combinatorial base of our semistability proof. 
\begin{lemma}\label{tight}
Suppose $t_1$ and $t_2$ are vertices of depth $\bar d\geq \delta$ in a horoball $H$ of $X$. Then for each $i\in \{1,2\}$, there is a geodesic $\gamma_i$  from $\ast$ to $t_i$ such that  $\gamma_i$ has the form $(\eta_i, \alpha_i,\tau_i, \beta_i)$, where the end point $x_i$ of $\eta_i$ is the first point of $\gamma_i$ in $H(\bar d)$, $\alpha_i$ and $\beta_i$ are vertical and of the same length in $H^{\bar d}$ and $\tau_i$ is horizontal of length $\leq 3$. Furthermore $d(x_1,x_2)\leq 2\delta +1$. 
\end{lemma}
\begin{proof}
The proof of the first part follows directly from Lemma \ref{geo}. For the last part, simply consider the geodesic triangle formed by the geodesics $\eta_1$, $\eta_2$ and the geodesic between $x_1$ and $x_2$. The two vertical segments of this last geodesic cannot be longer than $\delta-1$ since the triangle is $\delta$-thin. 
\end{proof}

We call a path of the form $(\eta_i, \alpha_i,\tau_i, \beta_i)$ a {\it $\mathcal D(t_i)$-standard geodesic} from $\ast$ to $t_i$. The next lemma basically says if $\mathcal D(t)=\bar d\geq \delta$, $H$ is a horoball containing $t$ and $z$ is a closest point of $H(\bar d)$ to $\ast$, then the ($X$) distance $d(t,\ast)$ is approximately the distance from $\ast$ to $z$ plus the distance (in $H^{\bar d}$) from $z$ to $t$.

\begin{lemma} \label{Base} 
Suppose $t$ is a vertex of a horoball $H$ of $X$ such that $\mathcal D(t)=\bar d\geq \delta$. 
Let $z$ be a point of $H(\bar d)$ closest to $\ast$. Then  
$$d(\ast,t)\leq d(\ast,z)+d(z,t)\leq 2\delta+1+d(\ast, t)$$ 
\end{lemma} 
\begin{proof}
The first inequality follows from the triangle inequality. Let $(\eta, \alpha,\tau, \beta)$ a $\bar d$-standard geodesic from $\ast$ to $t$. Let $p$ be the end point of $\eta$. By Lemma \ref{tight}, $d(z,p)\leq 2\delta+1$ and so by the triangle equality 
$$d(z,t)\leq d(z,p)+d(p,t)\leq 2\delta+1+d(p,t)$$
Also, $d(\ast,z)\leq d(\ast,p)$. Adding inequalities, 
$$d(\ast,z)+d(z,t)\leq d(\ast,p)+2\delta+1+d(p,t)\leq 2\delta +1+d(\ast,t)$$ 
\end{proof}

The next two results follow directly from the definition of Groves-Manning space.
\begin{lemma} \label{SubBall1} 
Suppose $z$ and $t$ are vertices of a horoball $H$ of $X$, $\mathcal D(z)=\mathcal D(t)=m (\geq \delta)$ and  $d^m(z,t)\leq 2^n$. Then $d(z,t)\leq 2n$, and so $\hat B_{2^n}(z)\subset B_{2n}(z)$. 
\end{lemma}

\begin{lemma}\label{SubBall2} 
Suppose $z$ and $t$ are vertices of a horoball $H$, $\mathcal D(z)=\mathcal D(t)=m(\geq \delta)$ and $d(z,t)\leq 2n+3$ (respectively $2n+2$). Then 
$d^m(z,t)\leq 3(2^n)$ (respectively $2^{n+1}$), and so $B_{2n+3}(z)\cap H(m)\subset \hat B_{3(2^n)}(z)$ (respectively $B_{2n+2}(z)\cap H(m)\subset \hat B_{2^{n+1}}(z)$). In particular, for any integer $k>0$, $B_k(z)\cap H(m)\subset \hat B_{2^{[{k\over 2}]+1}} (z)$.
\end{lemma}

\begin{lemma}\label{NoLinDeadEnd} 
Suppose $G$ is a finitely generated group with finite generating set $S$. Let $\Gamma$ be the Cayley graph of $G$ with respect to $S$. Then for any integer $r>0$ and any point $x\in \Gamma-B_{2r}(\ast)$ there is a geodesic ray at $x$ which avoids $B_r(\ast)$.
\end{lemma}
\begin{proof}
Let $L$ be a geodesic line in $\Gamma$. By translation, we may assume $x$ is a vertex of $L$. Now consider the two geodesic rays $q^+$ and $q^-$ in opposite directions on $L$, both with initial point $x$. If both $q^+$ and $q^-$ intersect $B_r(\ast)$ at say $x^+$ and $x^-$ respectively, then since $x\not\in B_{2r}(\ast)$ both of $x^+$ and $x^-$ are of distance at least $r+1$ from $x$. This means the distance from $x^+$ to $x^-$ is at least $2r+2$. But, since both belong to $B_r(\ast)$ they are at most distance $2r$ from one another - the desired contradiction. 
\end{proof}

Suppose $x,y\in X$ are points of a horoball $H$ and $\mathcal D(x)=\mathcal D(y)=k$. Any path (without backtracking) in the $k$-horoball $H^k$ between $x$ and $y$ can be written as 
$$\psi=(\tau_0, \alpha_1,\tau_1,\alpha_2,\tau_2,\ldots, \alpha_n\tau_n)$$ 
where for each $i$, $\alpha_i$ is non-trivial  and vertical and $\tau_i$ is horizontal. With the possible exception of $\tau_0$ and $\tau_n$, $\tau_i$ is non-trivial.  
If $\tau_1$ is the edge path $(e_1,\ldots, e_m)$, let $\alpha_{1,j}$ be the vertical path from the end point of $e_j$ to $H(k)\subset H^k$. Also let $\alpha_1^{-1}=\alpha_{ 1,0}$. Next let $\gamma_{1,j}$ be a shortest edge path in $H(k)$ from the end point of $\alpha_{1,j-1}$ to the end point of $\alpha_{1,j}$ for all $j\in \{1,\ldots, m\}$. Let $\gamma_1=(\gamma_{1,1}, \ldots, \gamma_{1,m})$ and notice that for each vertex $v$ of $\gamma_1$ the vertical edge path of length $|\alpha_1|$ is within $1$ (horizontal) unit of a vertex of one of the $e_i$. For each $\tau_i$ similarly construct $\gamma_i$. The path $\gamma=(\gamma_1,\ldots, \gamma_n)$ is a {\it projection} of $\psi$ into $H(k)$.  

\begin{lemma}\label{Proj} 
Suppose $H$ is a horoball in $X$, $x,y\in H(k)$  and 
$$\psi=(\alpha_1,\tau_1,\alpha_2,\tau_2,\ldots, \alpha_n\tau_n, \alpha_{n+1})$$ is a path in $H^k$ between $x$ and $y$, where for each $i$, $\alpha_i$ is non-trivial  and vertical and $\tau_i$ is nontrivial and horizontal. 
If $\gamma$ is a projection of $\psi$ into $H(k)$ then each vertical line at a vertex of $\gamma$ passes within $1$ horizontal unit of a vertex of one of the $\tau_i$. 
\end{lemma}

The next lemma combines with Lemma \ref{Base} to show that if $\psi$ is an edge path in $X-B_r(\ast)$, has image in $H^{\hat d}$ ($\bar d\geq \delta$) for some horoball $H$, begins and ends in $H(\bar d)$  and the image of $\psi$ is ``far" from a vertex of $H(\bar d)$ closest to $\ast$ then $\psi$ projects to a path in $H(\bar d)$ that avoids $B_{r-(2\delta+1)}(\ast)$. The subsequent Lemma \ref{CaseN2}, shows that if $\psi$ has a vertex ``close" to a vertex of $H(\bar d)$ closest to $\ast$, then $\psi$ can be replaced by a path in $H(\bar d)$ that avoids $B_{r-(2\delta+5)}(\ast)$. These are the two main lemmas of the paper. In the proof of our main theorem, we only need these lemmas with $\bar d=\delta$. In the proof of  Theorem \ref{DD+}, we will need the general version of these lemmas.  Lemma \ref{CaseN1} has a variation (Lemma \ref{Case1}) that is in some sense stronger and of separate interest. We only need the version that immediately follows in order to prove our main theorem,  and so we delay the introduction of Lemma \ref{Case1} to section \S \ref{FC}. 

\begin{lemma} \label{CaseN1} 
Suppose $H$ is a horoball, $\bar d$ is an integer $\geq \delta$, $x\ne y$ are vertices of $H(\bar d)$, $z$ is a closest vertex of $H(\bar d)$ to $\ast$, $\psi$ is a path in $H^{\bar d}-B_r(\ast)$ between $x$ and $y$ that only intersects $H(\bar d)$ at $x$ and $y$ and $L=|\psi|$. 

If $\gamma$ is a projection of $\psi$ to $H(\bar d)$ and each vertex $v$ of $\psi$ is such that $d(v,z)>L+2$, then each vertex of $\gamma$ is at distance greater than $L+2$ from $z$. 
Furthermore,  the image of $\gamma$ avoids $B_{r-(2\delta+1)}(\ast)$. 
\end{lemma}
\begin{proof} 
By Lemma \ref{geo}, $H^{\bar d}$ is convex. Let $p$ be a vertex of $\gamma$ and $(\alpha_p,\tau_p,\beta_p)$ a geodesic (in $H^{\bar d}$) from $z$ to $p$ where $\alpha_p$ and $\beta_p$ are vertical of the same length and $|\tau_p|\leq 3$. Let $y$ be the end point of $\tau_p$. By Lemma \ref{Proj}, there is a vertical segment that begins at $p$ and ends at most 1 (horizontal) unit from a vertex $v$ of $\psi$ (and $\mathcal D(v)\geq \bar d+1$). If $y$ is on that vertical line segment (Figure 1), then there is a path $\rho$ from $v$ to $z$ that begins with a horizontal edge from $v$ to a vertex $w$ on the vertical segment, followed by a vertical segment from $w$ to $y$, followed by $\tau_p^{-1}$, followed by $\alpha_p^{-1}$. 

 The sum of the lengths of the two vertical segments of $\rho$ is less than or equal to ${L-1\over 2}$ (since $v$ is on $\psi$), and so 
$$ L+2<d(v,z)\leq {L-1\over2}+4={L\over 2}+{7\over2}\hbox{ implying }$$ 
$$2L+4<L+7 \hbox{ and  } L<3,$$
which is impossible since $x\ne y$.

\vspace {.5in}
\vbox to 2in{\vspace {-2in} \hspace {-1.3in}
\hspace{-1 in}
\includegraphics[scale=1]{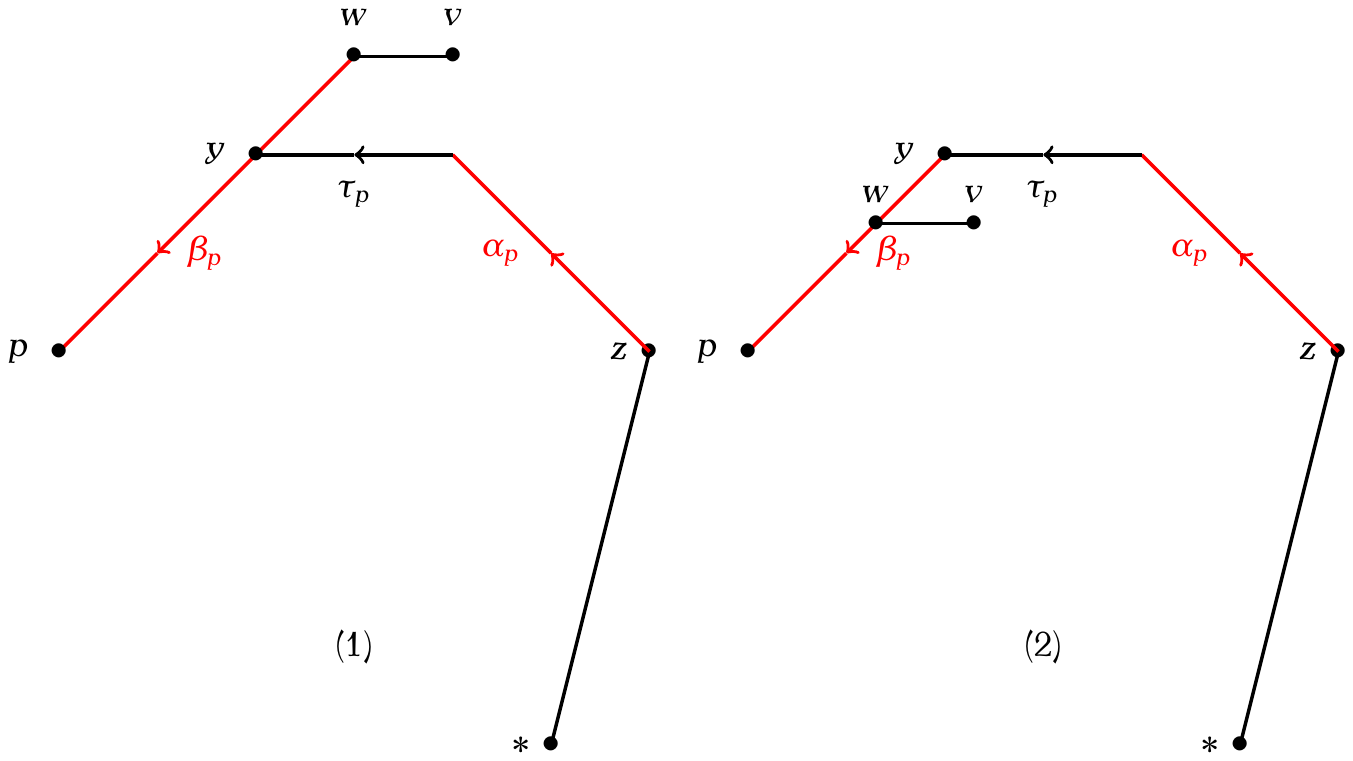}
\vss }

\vspace{1 in}

\centerline{Figure 1}

\medskip

Instead, $y$ is on the vertical line at $p$ and the vertical line segment from $p$ to $y$ contains a vertex $w$ (other than $y$) within 1 horizontal unit of a vertex $v$ of $\psi$ (Figure 1.2). Now 
$$d(z,p)=d(p,w)+d(w,y)+|\tau_p| +|\alpha_p|\hbox{ and}$$  
$$L+2<d(v,z)\leq 1+d(w,y)+|\tau_p| +|\alpha_p|$$ 
Since $d(p,w)\geq 1$, $d(z,p)>L+2$. Completing the first part of the lemma.

Since the depth of $v$ (and hence the depth of $w$) is at least $\bar d+1$,
$$d(p,z) \geq 1+d(w,z)\geq d(v,z)$$
Combining this inequality with Lemma \ref{Base} and the triangle inequality, we have for each vertex $p$ of $\gamma$:
$$ d(p,\ast)+2\delta+1\geq d(p,z)+d(z,\ast)\geq d(v,z)+d(z,\ast)\geq d(v,\ast)>r$$
 So $d(p,\ast)> r-2\delta -1$.
\end{proof}

\begin{remark} \label{RN1}
Assume all of the hypotheses of Lemma \ref{CaseN1}, except $d(v,z)>L+2$ for all vertices $v$ of $\gamma$.  Consider the integer $k=r-d(z,\ast)$, so that $r=d(z,\ast)+k$. Note that if $v\in B_{L+2}(z)$ then $d(v,\ast)\leq d(v,z)+d(z,\ast)\leq L+2+d(z,\ast)=L+2+r-k$. So if $k\geq L+2$  then $d(v,\ast) \leq r$ and  $B_{L+2}(z)\subset B_r(\ast)$.
Since $\psi$ has image in $X-B_r(\ast)$, each vertex $v$ of $\psi$ is such that $d(v,z)>L+2$. Then the hypotheses (and hence the conclusion) of Lemma \ref{CaseN1} is satisfied. 
\end{remark}

If we consider Lemma \ref{CaseN1}  and Lemma \ref{CaseN2} simultaneously,  we see that any path $\psi$ connecting points $x$ and $y$ (as described in these lemmas) must satisfy the conclusion of one of these lemmas. By Remark \ref{RN1}, the situation where $k=r-d(z,\ast)\geq L+2$ (and hence every vertex of $\psi$ is of distance $>L+2$ from the special vertex $z$), is covered by Lemma \ref{CaseN1}.  The case when some vertex of $\psi$ is within $L+2$ of $z$, (and hence by Remark \ref{RN1}, $k<L+2$) is covered by Lemma \ref{CaseN2}.

\begin{lemma} \label{CaseN2} 
Suppose $L$ is a fixed integer, $\bar d$ an integer $\geq \delta$, $x,y\in H(\bar d)-B_r(\ast)$ for some horoball $H$, $\psi$ is a path in $H^{\bar d}-B_r(\ast)$ between $x$ and $y$, the image of $\psi$ intersects $H(\bar d)$ only at $x$ and $y$, $z$ is a closest point of $H(\bar d)$ to $\ast$, the length of $\psi$ is $L$, and some vertex $v$ of $\psi$ is such that $d(v,z)\leq L+1$.  Then there is an integer $F_{\ref{CaseN2}}(L)$ (independent of $r$) and a path in $H(\bar d)$ of length $\leq F_{\ref{CaseN2}}$ from $x$ to $y$ that avoids $B_{r-(2\delta+5)}(\ast)$. 
\end{lemma}
\begin{proof}
First observe that by Remark \ref{RN1}, $k=r-d(z,\ast)<L+2$. Since $|\psi|=L$ and $d(v,z)\leq L$, both $d(x,z)$ and $d(y,z)$ are $\leq 2L+1$. By  Lemma \ref{SubBall2}, $x,y\in \hat B_{2^{L+1}}(z)(\subset H(\bar d))$. If $k=r-d(z,\ast)<0$  then (since $z$ is a closest point of $H(\bar d)$ to $\ast$) no point of $H(\bar d)$ belongs to $B_{r}(\ast)$. Hence any path in $H(\bar d)$ between $x$ and $y$ will avoid $B_{r}(\ast)$. In particular, a projection of $\psi$ to $H(\bar d)$ will have length $\leq 2^{L-1\over 2}$ and avoid $B_r(\ast)$.
Let $F_{-1}(L)=2^{L-1\over 2}$ (independent of $r$). 



 Now we need only consider the cases where $k\in \{0,1,\ldots, L+1\}$. For simplicity, assume $k$ is a positive even number. If $d(t,z)\leq k=r-d(\ast,z)$ then $r\geq d(\ast, z)+d(z,t)\geq d(\ast, t)$,  and so $t\in B_r(\ast)$. In particular, $B_k(z)\subset B_r(\ast)$. By Lemma \ref{SubBall1}
$$\hat B_{2^{k\over 2}}(z)\subset B_k(z)\subset B_r(\ast)$$
so that $x,y \in \hat B_{2^{L+1}}(z) - \hat B_{2^{k\over 2}}(z)$.  For every horoball $H$ in $X$, $H(\bar d)$ is a Cayley graph for some peripheral subgroup and hence is 1-ended. Lemma 2.6 provides a path $\gamma(x,y)$  from $x$ to $y$, with image in $H(\bar d)-\hat B_{2^{{k\over 2}-1}}(z)$.

Next, we show that $B_{r-(2\delta+5)}(\ast)\cap H(\bar d)\subset \hat B_{2^{{k\over 2}-1}}(z)$ 
 (in order to ensure that $\gamma(x,y)$ avoids $B_{r-(2\delta+5)}(\ast)$). By Lemma \ref{SubBall2}, $B_m(z)\cap H(\bar d)\subset \hat B_{2^{[{m\over 2}]+1}} (z)$ and so $B_{k-4}(z)\cap H(\bar d)\subset \hat B_{2^{{k\over 2}-1}}(z)$. It suffices to show that  $B_{r-(2\delta+5)}(\ast)\cap H(\bar d)\subset B_{k-4}(z)\cap H(\bar d)$. 

Say $t\in B_r(\ast)\cap H(\bar d)$ and let $p$ be the first vertex of $H(\bar d)$ on a geodesic from $\ast$ to $t$. By Lemma \ref{tight}, $d(z,p)\leq 2\delta+1$. Then $r\geq d(t,\ast)=d(\ast,p)+d(p,t)$ so that $k=r-d(\ast,z)\geq r-d(\ast, p)\geq d(p,t)$, and 
$$d(z,t)\leq d(z,p)+d(p,t)\leq 2\delta+1+k$$
In particular, $t\in B_{2\delta+1+k}(z)\cap H(\bar d)$ and 
$$B_r(\ast)\cap H(\bar d)\subset B_{2\delta+1+k}(z)\cap H(\bar d)=B_{2\delta+1+r-d(\ast,z)}(z)\cap H(\bar d)$$

This last formula is true for all $r$, and replacing $r$ by $r-2\delta-5$ gives
$$(1)\ \ \ B_{r-2\delta-5}(\ast)\cap H(\bar d) \subset B_{r-4-d(\ast,z)}(z)\cap H(\bar d)=B_{k-4}(z)\cap H(\bar d)\subset \hat B_{2^{{k\over 2}-1}}(z).$$
 
Now, for $x,y\in \hat B_{2^{L+1}}(z)-\hat B_{2^{k\over 2}}(z)$, $\gamma(x,y)$ is a path from $x$ to $y$ with image in $H(\bar d)-\hat B_{2^{{k\over 2} -1}}(z)$ and so $\gamma(x,y)$ avoids $B_{r-2\delta-5}(\ast)$. 
Let $F_k(L)$ be the length of a longest path $\gamma(x,y)$ for $x,y\in \hat B_{2^{L+1}}(z)-\hat B_{2^{k\over 2}}(z)$. We will show that for any $G$-translate $H'(\bar d)$ of $H(\bar d)$ with vertex $z'\in H'(\bar d)$ closest to $\ast$ 
and such that $k=k'=r'-d(z',\ast)$, the element $g\in G$ such that $gz=z'$ can be used to translate the paths $\gamma(x,y)$ to paths connecting any pair of points in $\hat B_{2^{L+1}}(z')-\hat B_ {2^{k\over 2}}(z')$ so that $g\gamma(x,y)$ avoids $B_{r'-(2\delta+5)}(\ast)$.

Say $H'(\bar d)$ is a $G$-translate of $H(\bar d)$, $z'$ is a closest point of $H'(\bar d)$ to $\ast$, $k=k'=r'-d(z',\ast)$ and $gz=z'$ (so $gH(\bar d)=H'(\bar d)$). Note that $g$ maps $\hat B_s(z)$ to $\hat B_s(z')$ for any integer $s$. So if $x',y'\in \hat B_{2^{L+1}}(z')-\hat B_{2^{k'\over 2}}(z')$ then there is a path connecting $x'$ and $y'$  of length $\leq F_k$ in $H'(\bar d)-\hat B_{2^{{k'\over 2}-1}}(z')$. By the same argument that establishes equation (1),  $B_{r'-(2\delta+5)}(\ast)\cap H'(\bar d)\subset \hat B_{2^{{k'\over 2}-1}}(z')$, so the paths connecting $x'$ and $y'$ avoid  $B_{r'-(2\delta+5)}(\ast)\cap H'(\bar d)$.
 Lastly, let $F_{\ref{CaseN2}}(L)=\max\{F_{-1},F_0,\ldots, F_{L+1}\}$. 
 \end{proof}

 \section{Proof of the Main Theorem}\label{MT} 
 
The proof of Theorem \ref{Main} is completed in this section. The two papers \cite{GMa08}-Groves Manning and \cite{DG08}-Dahmani-Groves along with Bowditch's many results are used. We list a few results:

Theorem 3.33 of \cite{GMa08} insures that there are geodesic lines in $X$ that are axes for elements of $G$. A mild alteration of the proofs of Lemma 3.32 and Theorem 3.33, ensures that such lines intersect $Y$. This gives 

\begin{theorem} \label{GM3.33} 
\cite{GMa08} There is an infinite order element $g\in G$ so that if $\rho$ is a geodesic from $\ast$ to $g\ast$, then the line $(\ldots, g^{-1}\rho, \rho, g\rho,g^2\rho,\ldots)$ is a bi-infinite geodesic that has image in 
$\mathcal D^{-1}[0,19\delta]$.
\end{theorem}

We need a result of \cite{DG08}, but some background first. In (\cite{DG08}, \S 2.3) several constants are defined. There $\delta$ is any positive integer such that $X$ is $\delta$-hyperbolic and $C=3\delta$. Following \cite{DG08}, define 
$$M=6(C+45 \delta)+2\delta+3=6(45 \delta)+20 \delta+3\hbox{ and } K=2M.$$ 

Two conditions on points in $X$ are defined. First:

Given $\epsilon\geq 0$, and two points $x,y\in X$, say that $x$ and $y$ satisfy $\star_\epsilon$ if 
$$\star_\epsilon :|d(\ast,x)-d(\ast, y)|\leq \epsilon \hbox{ and }d(x,y)\leq M.$$

Given an integer $N$, say $x,y\in X$ satisfying $\star_\epsilon$, satisfy condition $  \ddagger (\epsilon,N)(x,y)$  if there is a path of length at most $N$ from $x$ to $y$ in the compliment of the ball $B_{m-48  \delta}(\ast)$ where $m=min\{d(\ast, x),d(\ast,y)\}$. 

\begin{lemma} \label{DG4.2} 
(\cite{DG08}, Lemma 4.2) If the boundary $\partial X$ is connected, and has no global cut point, then there is an integer $  N$ such that $  \ddagger(10 \delta,  N)(x, y)$ holds for all $x, y \in X_{ K}=\mathcal D^{-1}([0,  K])$ satisfying $\star_{10 \delta}$.
\end{lemma}

Note that the path between $x$ and $y$ of this last lemma has image in $X$ (not $X_K$). Later we improve this result (see Theorem \ref{DD+}) to get a path in $X_K$.

\begin{lemma} \label{AlmostExt} 
For any vertex $v$ of $Y$ there is a geodesic ray $r_v$ in $X$ such that $r_v(0)=\ast$ , there is $t_v\in [0,\infty)$ such that $d(r_v(t_v),v)\leq \delta$ and $r_v|_{[t_v,\infty)}$ has image in $\mathcal D^{-1}([0,21\delta])$.
\end{lemma}
\begin{proof}
Let $\ell$ be the line of Theorem \ref{GM3.33}. By definition the vertices of $Y$ are the elements of $G$, so that $v\ell$ is a geodesic line through $v$ and with image in $\mathcal D^{-1}[0,19\delta]$. By an elementary local finiteness argument, there are geodesic rays $r_v$ and $s_v$ converging to the two ends of $g\ell$ such that for any integer $n\in [0,\infty)$,  $r_v|_{[0,n]}$ can be extended to a geodesic ending at a vertex of $\ell_v$ (and the same is true for $s_v$). The ideal geodesic triangle formed by $r_v$, $s_v$ and $v\ell$ is $\delta$-thin. In particular, $v$ is within $\delta$ of some vertex of $r_v$ or $s_v$. Assume there is $t_v\in [0,\infty)$ such that $d(r_v(t_v),v)\leq \delta$. Since for any $n\geq t_v$, $r_v|_{[0,n]}$ can be extended to a geodesic ending on $g\ell$, $d(r_v(n),v\ell)\leq 2\delta$. As $v\ell$ has image in  $\mathcal D^{-1}([0,19\delta])$, $r_v|_{[t_v,\infty)}$ has image in $\mathcal D^{-1}([0,21\delta])$.
\end{proof}

\begin{proof} {\it (of the main theorem)}
First assume that $G$ is finitely presented, so that $Y$ is simply connected. Let $g$ and $\rho$ be as in Lemma \ref{GM3.33} and say $\rho$ is the edge path $(e_1,\ldots, e_{L_0})$ so that $|\rho|=L_0$. Say the consecutive vertices of $\rho$ are $\ast=v_0,\ldots, v_{L_0}$. If $v_i\in Y$, then let $y_i=v_i$. Otherwise, let $y_i\in Y$ be the end point of a vertical path (of length $\leq 19\delta$) from $v_i$ to $Y$. Let $\hat\tau_i$ be a shortest path in $Y$ from $y_{i-1}$ to $y_i$. 

\vspace {.5in}
\vbox to 2in{\vspace {-2in} \hspace {-1.3in}
\hspace{-1 in}
\includegraphics[scale=1]{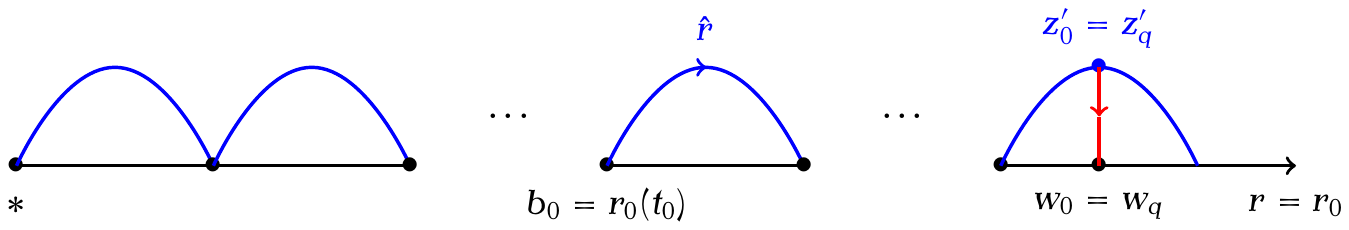}
\vss }

\vspace{-1.2 in}

\centerline{Figure 2}

\medskip

Let $\hat \rho$ be the $Y$-path $\hat \rho=(\hat \tau_1,\ldots,\hat \tau_{L_0})$. Then $\hat r=(\hat\rho,g\hat \rho, g^2\hat \rho,\ldots)$ is a proper edge path ray in $Y$ beginning at $\ast$ and we use $\hat r$ as a base ray in our semistability argument. Notice that the $X$-geodesic ray $r=(\rho, g\rho, g^2\rho,\ldots)$ tracks $\hat r$ and for any vertex $v$ of $r$, there is a vertical path of length $\leq 19\delta$ from $v$ to a point of $\hat r$. Furthermore, $r$ and $\hat r$ share the vertices $\ast, g, g^2,\ldots$ (where $g^n=g^n\ast$). (See Figure 2.)

All of our compacts sets will be finite subcomplexes of $X$ and all paths will be edge paths. Given a compact set $C$ in $Y$, our goal is to find a compact set $D$ in $Y$ with the following semistability property (see Theorem \ref{EquivSS}(2)): For any third compact set $E\subset Y$ and $Y$-loop $\beta$ based on $\hat r$ and with image in $Y-D$, show that $\beta$ is homotopic rel$\{\hat r\}$ to a loop in $Y-E$ by a homotopy in $Y-C$. We will define a set of integers $N_0$, $N_1$, $N_2$, $N_3$ and $N_4$. Our compact set $D$ will be $B_{N_4}(C)\cap Y$, the ball in $X$ of radius $N_4$ about $C$ intersected with $Y$.      

Suppose $\beta$, with consecutive vertices $b_0=r(t_0), b_1, b_2,\ldots, b_{q-1},b_q=b_0$, is a loop based at a common vertex $b_0$ of $\hat r$ and $r$. Let $r=r_0=r_q$.  By Lemma \ref{AlmostExt}, for each $i\in \{1,\ldots, q-1\}$, there is a geodesic ray  $r_i$ at $\ast$ such that for some integer $t_i\in [0,\infty)$, $d(r_i(t_i),b_i)\leq \delta$  and  $r_i|_{[t_i,\infty)}$ has image in $\mathcal D^{-1}([0,21\delta])$. 
Define $\tau_i$ to be a path from $b_i$ to $r_i(t_i)$ of length $\leq \delta$. (See Figure 3.) Since $b_0=r_0(t_0)$, $\tau_0$ is the trivial path. Note that $d(r_{i-1}(t_{i-1}),r_i(t_i))\leq 2\delta+1$. Let $w_i=r_i(t_i+3(45 \delta)+9 \delta)$ for all $i$. Then 
$$d(w_{i-1}, w_i)\leq 2[3(45 \delta)+9  \delta] +2\delta +1= 6(45  \delta)+20\delta +1<  M.$$

\vspace {.5in}
\vbox to 2in{\vspace {-2in} \hspace {-1.3in}
\hspace{-.5 in}
\includegraphics[scale=1]{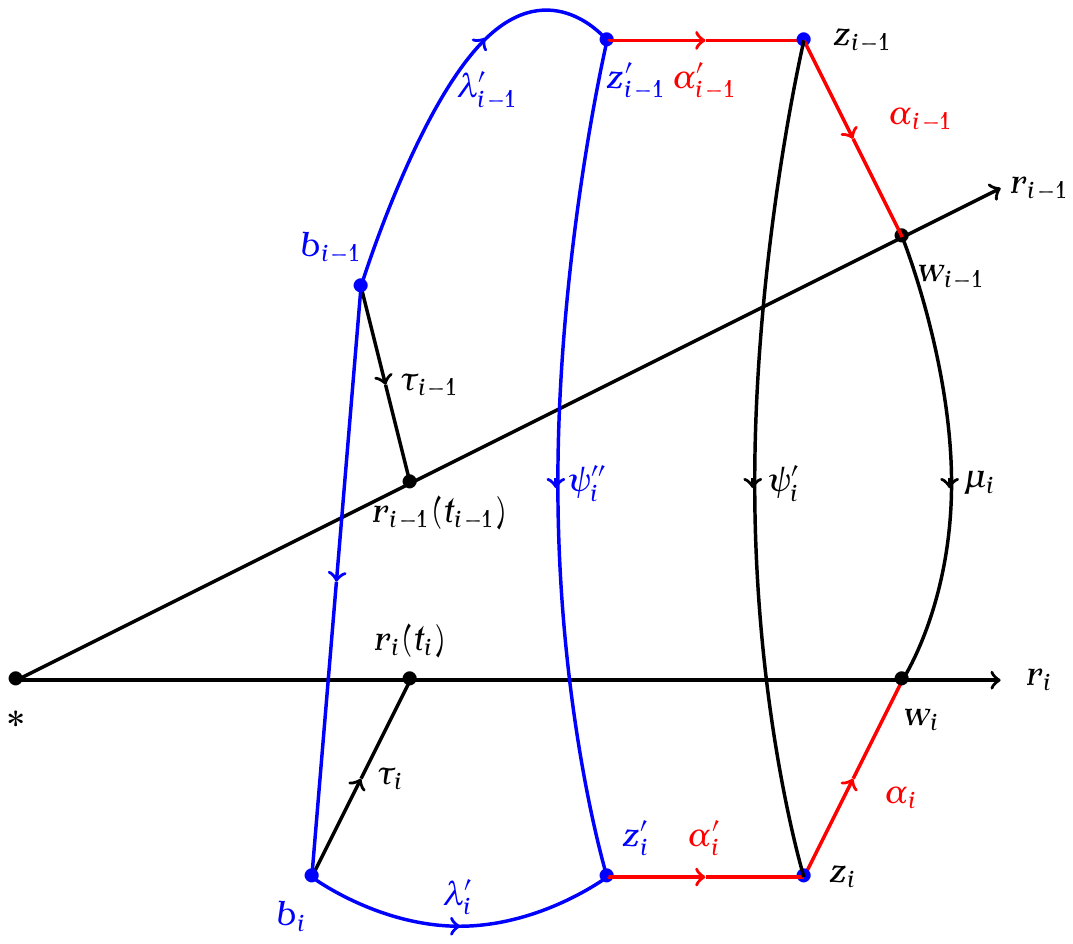}
\vss }

\vspace{2 in}

\centerline{Figure 3}

\medskip

Since $d(r_{i-1}(t_{i-1}), r_i(t_i))\leq 2\delta +1$ and $d(r_{i}(t_{i}), w_i)$ is the same for all $i$, $|d(w_{i-1},\ast)-d(w_i,\ast)|\leq 2\delta +1<10  \delta$. As $w_i\in \mathcal  D^{-1}([0,21\delta])$ and $21 \delta<  K$, Lemma \ref{DG4.2} implies $w_{i-1}$ and $w_i$ satisfy $  \ddagger(10 \delta, N_0) (w_{i-1},w_i)$ for some positive integer $N_0$. That means
there is a path $\mu_i$ in $X$ from $w_{i-1}$ to $w_i$   of length $\leq   N_0$ in the compliment of the ball $B_{m-48  \delta}(\ast)$ where $m=min\{d(\ast, w_{i-1}),d(\ast,w_i)\}$. As $\mathcal D(w_i)\leq 21\delta$ there is a vertical path $\alpha_i$ of length $\leq 20\delta$ from a point $z_i\in \mathcal D^{-1}(\delta)$ to $w_i$. (By the definition of $\hat r$, the vertical path from $w_0$ to $Y$ ends at a point in the image of $\hat r$ and has length $\leq 19\delta$.) Hence the vertical path from $w_0$ to the point $z_0\in \mathcal D^{-1}(\delta)$ is of length $\leq 18\delta$. Let $\alpha_0$ be the vertical path from $z_0$ to $w_0$. {\it The path $\psi_i=(\alpha_{i-1},\mu_i,\alpha_i^{-1})$ begins and ends in $\mathcal D^{-1}(\delta)$, has image in $X-B_{m-68  \delta}(\ast)$ and length $\leq N_0+40\delta$.} The path $\psi_i$ can be decomposed into subpaths where each either has image in $\mathcal D^{-1}([0,\delta])$, 
or satisfies the conclusion of either Lemma \ref{CaseN1} or  Lemma \ref{CaseN2}. Let $\lambda$ be one of these subpaths. 

If $\lambda$ satisfies the conclusion of Lemma \ref{CaseN1}, then it has a projection $\lambda'$ (to $\mathcal D^{-1}(\delta)$) with image in $X-B_{m-68  \delta-(2\delta+1)}(\ast)$. The length of $\lambda$, and in particular the length of $\lambda'$, is bounded by a constant only depending on $|\psi_i|\leq N_0+40  \delta$.

If $\lambda$ satisfies the conclusion of Lemma \ref{CaseN2}, then there is path $\lambda'$ (in $\mathcal D^{-1}(\delta)$) connecting the end points of $\lambda$, with image in $X-B_{m-68  \delta-(2\delta+5)}(\ast)$, and again the length of $\lambda'$ is bounded by a constant only depending  on the length of $\lambda$, and in particular, by a constant depending only on $|\psi_i|\leq N_0+40  \delta$. Replacing each of the $\lambda$ subpaths by the corresponding $\lambda'$ paths in $\mathcal D^{-1}(\delta)$ gives us a path $\psi_i'$ in $\mathcal D^{-1}([0,\delta])$ from $z_{i-1}$ to $z_i$, of length $\leq N_1$ (a constant depending only on the number $N_0+40 \delta$). {\it Then $|\psi_i'| \leq N_1$ and the image of $\psi'$ is in $X-B_{m-68\delta-(2\delta+5)}(\ast)$.}

Let $\alpha_i'$ be the vertical path of length $\delta$ from $z_i'\in Y$ to $z_i$. Note that $z_0'$ is in the image of $\hat r$. Let $\psi_i''$ be a projection of $\psi_i'$ to $Y$. Again, the length of $\psi''$ has length $\leq N_2$ (a constant depending only on the number $N_0+40 \delta$). By Lemma \ref{Proj}, each vertex of $\psi_i''$ is within $\delta+1$ of a vertex of $\psi_i'$. {\it Hence, $|\psi_i''|\leq N_2$  and the image of $\psi_i''$ is in $X-B_{m-68\delta-(2\delta+5)-(\delta +1)}(\ast)=X-B_{m-71\delta-6}(\ast)$.}

Our goal is to show that the edges $[b_{i-1},b_i]$ are compatibly homotopic to the $\psi_i''$, and that each vertex of $\psi_i''$ is further from $\ast$ than either of $b_{i-1}$ or $b_i$. We must also show the homotopies that exchange the edges of $\beta$ for the $\psi_i''$ avoid the compact set $C$ (but first we must decide how large the yet to be defined $N_4$ must be). Once this is accomplished, then repeating the process moves the original path (rel $\{\hat r\}$) to a path outside of any preassigned compact set $E$ (by a homotopy avoiding $C$) to finish the proof.

First we show each vertex $v$ of $\psi_i''$ is further from $\ast$ than either $b_i$ or $b_{i-1}$. We have 
$$d(v,\ast)\geq m-71 \delta -6\hbox { where}$$
$$m=min\{d(\ast, w_{i-1}), d(\ast, w_i)\}=$$ 
$$min\{d(\ast, r_{i-1}(t_{i-1})), d(\ast, r_i(t_i))\}+3(45 \delta) +9 \delta$$
Then,
$$d(v,\ast)\geq min\{d(\ast, r_{i-1}(t_{i-1})), d(\ast, r_i(t_i))\}+73\delta-6$$
Since $d(r_i(t_i),b_i)\leq \delta$ for all $i$:
$$d(v,\ast)\geq min\{d(b_{i-1},\ast), d(b_i,\ast)\}+72 \delta-6$$
Since $d(b_{i-1},b_i)\leq 1$, $v$ is at least $72  \delta-7$ further from $\ast$ than either $b_{i-1}$ or $b_i$ is from $\ast$. 
We have shown:

(2) Each vertex of $\psi_i''$ is at least $72\delta-7$ further from $\ast$ than either $b_{i-1}$ or $b_i$ is from $\ast$.

Recall that $|\tau_i|\leq \delta$, $d(r(t_i),w_i)=3(45\delta)+9\delta$ and $|(\alpha_i',\alpha)|\leq 21\delta$. Consider the $X$ path $(\tau_{i}, [r_i(t_i), w_i], \alpha_i^{-1},\alpha_i'^{-1})$ from $b_i$ to $z_i'$ (both vertices in $Y$) of length $\leq 3(45 \delta)+31 \delta=L_1$.  
There is a positive integer $N_3$ only depending on $L_1$ so that for any edge path $\lambda$ in $X$ of length $\leq L_1$, that begins and ends in $Y$, there is an edge path $\lambda'$ of length $\leq N_3$ in $Y$ connecting the end points of $\lambda$. In particular, there is a $Y$-edge path $\lambda_i'$ from $b_i$ to $z_i$ of length $\leq N_3$. When $i=0$ or $i=n$ we want $\lambda_0'=\lambda_n'$ to be the subpath of $\hat r$ from $b_0=b_n$ to $z_0=z_n$. Since $b_0=r_0(t_0)$ and $d(r_0(t_0),w_0)=3(45\delta)+9\delta$, we can assume $N_3$ is larger than the length of the subpath of $\hat r$ from $b_0=r_0(t_0)\in im(\hat r)$  to $z_0=z_n$.

The loops $\ell_i=([b_{i-1},b_i], \lambda_i',  \psi_i''^{-1}, \lambda_{i-1}'^{-1})$ have image in $Y$ and are of length $\leq 1+2N_3+N_2$. Choose an integer $N_4$ so that for any loop of length $\leq 1+2N_3+N_2$ in $Y$ the loop is homotopically trivial by a homotopy in $B^Y_{N_4}(v)$ (the ball in $Y$ of radius $N_4$) for any vertex $v$ of the loop. Define the compact set $D$ to be $B_{N_4}(C)\cap Y$ (the ball in $X$ of radius $N_4$ about $C$, intersected with $Y$). Assume that $\beta$ is a loop in $Y$ based at a vertex of $r$ common to $\hat r$ and with image in $Y-D$. The homotopies killing the $\ell_i$ in $Y$ (with image in $St_Y^{N_4} (v)$ for any vertex $v$ of $\ell_i$) move each edge $[b_{i-1},b_i]$ to a path $\psi_i''$ by a homotopy in $Y$ whose image avoids the compact set $C$. These homotopies string together compatibly to move $\beta$ further out along $\hat r$. Continuing, $\beta$ is eventually moved outside of any compact subset $E$ of $Y$. We have shown that $Y$, and hence $G$ has semistable fundamental group at $\infty$.  

Note that in order to show that $Y$ has semistable fundamental group at $\infty$, the only homotopies we need are those that kill loops of length $\leq 1+2N_3+N_2$ in $Y$. If $G$ is only finitely generated (as opposed to finitely presented), then let $Y$ be obtained from $\Gamma(G,S)$ the Cayley graph of $G$ with respect to the generating set $S$, by attaching 2-cells at each vertex for each loop of length $\leq 1+2N_3+N_2$. The same proof (with only minor changes in the last paragraph) shows that $Y$ has semistable fundamental group at $\infty$. \end{proof}


\begin{theorem} \label{XKSS} 
Suppose $G$ is a 1-ended finitely generated group that is hyperbolic relative to the finitely generated proper subgroups $P_1,\ldots, P_n$. Let $X$ be a corresponding Groves-Manning space. If $Y$ has semistable fundamental group at $\infty$ then $X_k$ has semistable fundamental group at $\infty$ for all $k\geq 0$.
\end{theorem}
\begin{proof} 
Let $r$ and $s$ be proper rays at $\ast$ in $X_k$ for $k\geq 0$. Let $r'$ and $s'$ be projections of $r$ and $s$ to $Y$. Clearly, $r$ is properly homotopic to $r'$ and $s$ to $s'$. Since $r'$ and $s'$ are properly homotopic (by the semistablity of $Y$), a combination of proper homotopies shows $r$ is properly homotopic to $s$ in $X_k$.
\end{proof}


\section{An Improved $\ddagger$ Result} \label{SDD} 

 Throughout this section $G$ is a finitely generated group that is hyperbolic relative to a finite collection $\mathcal P$ of proper 1-ended finitely generated subgroups. Assume again that $X$ is a Groves-Manning space for $(G,\mathcal P)$. Say $X$ is $\delta_0$ hyperbolic and that $\delta\geq \delta_0$ is an integer. Lemmas \ref{CaseN1} and \ref{CaseN2} allow us to improve Lemma \ref{DG4.2} from a partial result about $X_{K(\delta)}$ to one completely about $X_{K(\delta)}$. More precisely, Lemma \ref{DG4.2} guarantees a path in 
 $X-B_{m-48\delta}(\ast)$ and we want to replace it by a path in $X_{K(\delta)}-B_{m-(50\delta+5)} (\ast)$.

First some observations. As Lemma \ref{DG4.2} gives a separate result for each $\delta\geq \delta_0$, the integers $K$, $M$ and $N$ will in fact be $K(\delta)$, $M(\delta)$ and $N(\delta)$. 
Consider $\delta\geq \delta_0$. Say $x,y\in X_{K(\delta)}$ satisfy $\star_{10\delta}$ and $\psi$ is a path in $X$ of length $\leq N(\delta)$ joining $x$ and $y$ in $X-B_{m-48\delta}(\ast)$ where $m=min\{d(\ast, x),d(\ast,y)\}$. Lemmas \ref{CaseN1} and \ref{CaseN2} allow us to replace the segments of $\psi$ that leave $X_{K(\delta)}$ by paths in $H({K(\delta)})$ (for some horoball $H$). The number of such segments and the lengths of the replacement paths are bounded in terms of $N(\delta)$ and the image of each such replacement path is in $X_{K(\delta)}-B_{m-(48\delta+2\delta+5)}(\ast)$. 

Recall the $\star_\epsilon$ and $\ddagger$ condition of \cite{DG08}, where for any $\delta\geq \delta_0$ we define
$$M(\delta)=6(45 \delta)+20 \delta+3\hbox{ and } K(\delta)=2M(\delta).$$

Given $\epsilon\geq 0$, and two points $x,y\in X$, say that $x$ and $y$ satisfy $\star_\epsilon$ if 
$$\star_\epsilon :|d(\ast,x)-d(\ast, y)|\leq \epsilon \hbox{ and }d(x,y)\leq M.$$ 
Note that if $\epsilon'\leq \epsilon$ and $x,y\in X$ satisfy $\star_{\epsilon'}$ then $x,y$ satisfy $\star_\epsilon$. 

Given an integer $N$, say $x,y\in X$ satisfying $\star_\epsilon$ satisfy condition $\ddagger (\epsilon,N)(x,y)$  if there is a path of length at most $N$ from $x$ to $y$ in  $X-B_{m-48  \delta}(\ast)$ where $m=min\{d(\ast, x),d(\ast,y)\}$. 

This induces the following definition.

Given an integer $N$, say $x,y\in X_{K(\delta)}$ satisfying $\star_\epsilon$ satisfy condition $\hat  \ddagger (\epsilon,N)(x,y)$  if there is a path of length at most $N$ from $x$ to $y$ in $X_{K(\delta)}- B_{m-50\delta-5}(\ast)$ where $m=min\{d(\ast, x),d(\ast,y)\}$. 

Lemma \ref{DG4.2} states that if the boundary $\partial X$ is connected, and has no global cut point, then there is an integer $N(\delta)$ such that $ \ddagger(10 \delta,  N)(x, y)$ holds for all $x, y \in X_{ K(\delta)}=\mathcal D^{-1}([0,  K(\delta)])$ satisfying $\star_{10 \delta}$. The path of length $\leq N$ connecting $x$ and $y$ and guaranteed  by Lemma \ref{DG4.2} has image in $X$. Lemmas \ref{CaseN1} and \ref{CaseN2}  enhance this result to give a path in $X_K$ (although $N$ will become larger and the image of the path may come $2\delta+5$ units closer to $\ast$ than the one given by $\ddagger(10\delta,N(x,y)$).

Combining these observations we have:

\begin{theorem} \label{DD+} 
Suppose $G$ is a finitely generated group that is relatively hyperbolic with respect to a finite collection of finitely generated 1-ended proper subgroups, the Groves-Manning space $X$ is $\delta_0$-hyperbolic and $\delta$ is an integer $\geq \delta_0$. If the boundary $\partial X$ is connected, and has no global cut point, then there is an integer $N(\delta)$ such that $\hat \ddagger(10 \delta,  N)(x, y)$ holds for all $x, y \in X_{ K(\delta)}(=\mathcal D^{-1}([0,  K(\delta)]))$ satisfying $\star_{10 \delta}$.
\end{theorem}

\section{An Interesting Technical Lemma} \label{FC} 

The following lemma seems more general than Lemma \ref{CaseN1}, since there is no restriction on the length of the path $\psi$ considered here, only a restriction on the depth of the path. The proof parallels that of Lemma \ref{CaseN1} and the second half is exactly the same.  

\begin{lemma} \label{Case1} 
Suppose $H$ is a horoball, $\bar d$ an integer $\geq \delta$, $x\ne y$ vertices of $H(\bar d)$ and $\psi$ is a path in $H^{\bar d}-B_r(\ast)$ between $x$ and $y$ that only intersects $H(\bar d)$ at $x$ and $y$.
Let $L$ be such that $\mathcal D(\psi)=[\bar d,L]$ (the depth of $\psi$). 

Suppose $\gamma$ is a projection of $\psi$ to $H(\bar d)$.  If $z$ is a closest vertex of $H(\bar d)$ to $\ast$ and each vertex $v$ of $\psi$ is such that $d(v,z)>L-{\bar d\over 2}+2$, then each vertex of $\gamma$ is at distance greater than $L-{\bar d\over 2} +2$ from $z$. 
Furthermore,  the image of $\gamma$ avoids $B_{r-(2\delta+1)}(\ast)$. 
\end{lemma}
\begin{proof} 
By Lemma \ref{geo}, $H^{\bar d}$ is convex. Let $p$ be a vertex of $\gamma$ and $(\alpha_p,\tau_p,\beta_p)$ a geodesic (in $H^{\bar d}$) from $z$ to $p$ where $\alpha_p$ and $\beta_p$ are vertical of the same length and $|\tau_p|\leq 3$. Let $y$ be the end point of $\tau_p$. By Lemma \ref{Proj}, there is a vertical segment that begins at $p$ and ends at most 1 (horizontal) unit from a vertex $v$ of $\psi$ (and $\mathcal D(v)\geq \bar d+1$). If $y$ is on that vertical line segment (Figure 1.1), then there is a path $\rho$ from $v$ to $z$ that begins with a horizontal edge from $v$ to a vertex $w$ on the vertical segment, followed by a vertical segment from $w$ to $y$, followed by $\tau_p^{-1}$, followed by $\alpha_p^{-1}$. 

 The sum of the lengths of the two vertical segments of $\rho$ is less than or equal to ${L-\bar d-1\over 2}$ (since $v$ is on $\psi$), and so 
$$ L-{\bar d\over 2}+2<d(v,z)\leq {L-\bar d-1\over2}+4={L\over 2}-{\bar d\over 2}+{7\over2}\hbox{ implying }$$ 
$$2L+4<L+7 \hbox{ and  } L<3,$$
which is impossible since $x\ne y$.  

Instead, $y$ is on the vertical line at $p$ and the vertical line segment from $p$ to $y$ contains a vertex $w$ (other than $y$) within 1 horizontal unit of a vertex $v$ of $\psi$ (Figure 1.2). Now 
$$d(z,p)=d(p,w)+d(w,y)+|\tau_p| +|\alpha_p|\hbox{ and}$$  
$$L-{\bar d\over 2}+2<d(v,z)\leq 1+d(w,y)+|\tau_p| +|\alpha_p|$$ 
Since $d(p,w)\geq 1$, $d(z,p)>L-{\bar d\over 2}+2$. Completing the first part of the lemma.

Since the depth of $v$ (and hence the depth of $w$) is at least $\bar d+1$,
$$d(p,z) \geq 1+d(w,z)\geq d(v,z)$$
Combining this inequality with Lemma \ref{Base} and the triangle inequality, we have for each vertex $p$ of $\gamma$:
$$ d(p,\ast)+2\delta+1\geq d(p,z)+d(z,\ast)\geq d(v,z)+d(z,\ast)\geq d(v,\ast)>r$$
 So $d(p,\ast)> r-2\delta -1$.
\end{proof}

\bibliographystyle{amsalpha}
\bibliography{paper}{}

\def\cprime{$'$}
\providecommand{\bysame}{\leavevmode\hbox to3em{\hrulefill}\thinspace}
\providecommand{\MR}{\relax\ifhmode\unskip\space\fi MR }
\providecommand{\MRhref}[2]{%
  \href{http://www.ams.org/mathscinet-getitem?mr=#1}{#2}
}
\providecommand{\href}[2]{#2}
\begin{thebibliography}{Bow99b}

\bibitem[BM91]{BM91}
Mladen Bestvina and Geoffrey Mess, \emph{The boundary of negatively curved
  groups}, J. Amer. Math. Soc. \textbf{4} (1991), no.~3, 469--481. \MR{1096169}

\bibitem[Bow99a]{Bow99A}
B.~H. Bowditch, \emph{Boundaries of geometrically finite groups}, Math. Z.
  \textbf{230} (1999), no.~3, 509--527. \MR{1680044}

\bibitem[Bow99b]{Bow99B}
\bysame, \emph{Connectedness properties of limit sets}, Trans. Amer. Math. Soc.
  \textbf{351} (1999), no.~9, 3673--3686. \MR{1624089}

\bibitem[Bow01]{Bow01}
\bysame, \emph{Peripheral splittings of groups}, Trans. Amer. Math. Soc.
  \textbf{353} (2001), no.~10, 4057--4082. \MR{1837220}

\bibitem[Bow12]{Bow12}
\bysame, \emph{Relatively hyperbolic groups}, Internat. J. Algebra Comput.
  \textbf{22} (2012), no.~3, 1250016, 66. \MR{2922380}

\bibitem[CM14]{CM2}
Gregory~R. Conner and Michael~L. Mihalik, \emph{Commensurated subgroups,
  semistability and simple connectivity at infinity}, Algebr. Geom. Topol.
  \textbf{14} (2014), no.~6, 3509--3532. \MR{3302969}

\bibitem[Dah03]{D03}
Fran\c{c}ois Dahmani, \emph{Combination of convergence groups}, Geom. Topol.
  \textbf{7} (2003), 933--963. \MR{2026551}

\bibitem[DG08]{DG08}
Fran\c{c}ois Dahmani and Daniel Groves, \emph{Detecting free splittings in
  relatively hyperbolic groups}, Trans. Amer. Math. Soc. \textbf{360} (2008),
  no.~12, 6303--6318. \MR{2434288}

\bibitem[DG13]{DG13}
Fran\c{c}ois Dahmani and Vincent Guirardel, \emph{Presenting parabolic
  subgroups}, Algebr. Geom. Topol. \textbf{13} (2013), no.~6, 3203--3222.
  \MR{3248731}

\bibitem[DGO17]{DGO17}
Fran\c{c}ois Dahmani, Vincent Guirardel, and Denis~V. Osin,
  \emph{Hyperbolically embedded subgroups and rotating families in gourps
  actiong onf hyperbolic spaces}, Mem. Amer. Math. Soc. \textbf{245} (2017),
  no.~1156, v+152. \MR{3589159}

\bibitem[DS78]{DydakS}
Jerzy Dydak and Jack Segal, \emph{Shape theory}, Lecture Notes in Mathematics,
  vol. 688, Springer, Berlin, 1978, An introduction. \MR{520227}

\bibitem[DtS05]{DS05}
Cornelia Dru\c~tu and Mark Sapir, \emph{Tree-graded spaces and asymptotic cones
  of groups}, Topology \textbf{44} (2005), no.~5, 959--1058, With an appendix
  by Denis Osin and Mark Sapir. \MR{2153979}

\bibitem[Dun85]{Dun85}
M.~J. Dunwoody, \emph{The accessibility of finitely presented groups}, Invent.
  Math. \textbf{81} (1985), no.~3, 449--457. \MR{807066}

\bibitem[Geo08]{G}
Ross Geoghegan, \emph{Topological methods in group theory}, Graduate Texts in
  Mathematics, vol. 243, Springer, New York, 2008. \MR{2365352}

\bibitem[GM85]{MR787167}
Ross Geoghegan and Michael~L. Mihalik, \emph{Free abelian cohomology of groups
  and ends of universal covers}, J. Pure Appl. Algebra \textbf{36} (1985),
  no.~2, 123--137. \MR{787167}

\bibitem[GM08]{GMa08}
Daniel Groves and Jason~Fox Manning, \emph{Dehn filling in relatively
  hyperbolic groups}, Israel J. Math. \textbf{168} (2008), 317--429.
  \MR{2448064}

\bibitem[GS]{GS17}
Ross Geoghegan and Eric Swenson, \emph{On semistability of {CAT(0)} groups},
  ArXiv: 1707.07061 [math.{GR}].

\bibitem[HR]{HR17}
C.~Hruska and K.~Ruane, \emph{Connectedness properties and splittings of groups
  with isolated flats}, ArXiv:1705.00784 [math.GR].

\bibitem[Mih86]{M4}
Michael~L. Mihalik, \emph{Semistability at {$\infty$} of finitely generated
  groups, and solvable groups}, Topology Appl. \textbf{24} (1986), no.~1-3,
  259--269, Special volume in honor of R. H. Bing (1914--1986). \MR{872498}

\bibitem[MT92]{MT1992}
Michael~L. Mihalik and Steven~T. Tschantz, \emph{Semistability of amalgamated
  products and {HNN}-extensions}, Mem. Amer. Math. Soc. \textbf{98} (1992),
  no.~471, vi+86. \MR{1110521}

\bibitem[Osi06]{Osin06}
Denis~V. Osin, \emph{Relatively hyperbolic groups: intrinsic geometry,
  algebraic properties, and algorithmic problems}, Mem. Amer. Math. Soc.
  \textbf{179} (2006), no.~843, vi+100. \MR{2182268}

\bibitem[PS09]{PS09}
Panos Papasoglu and Eric~L. Swenson, \emph{Boundaries and {JSJ} decompositions
  of {CAT}(0)-groups}, Geom. Funct. Anal. \textbf{19} (2009), no.~2, 559--590.
  \MR{2545250}

\bibitem[Swa96]{Swarup}
G.~A. Swarup, \emph{On the cut point conjecture}, Electron. Res. Announc. Amer.
  Math. Soc. \textbf{2} (1996), no.~2, 98--100. \MR{1412948}

\bibitem[Swe99]{Sw}
Eric~L. Swenson, \emph{A cut point theorem for {${\rm CAT}(0)$} groups}, J.
  Differential Geom. \textbf{53} (1999), no.~2, 327--358. \MR{1802725}

\end{thebibliography}

\end{document}